\documentclass[12pt,reqno]{amsart}
\usepackage[all]{xy}
\usepackage{amssymb,amsmath,amsthm,mathrsfs}
\usepackage{cases}
\usepackage{enumerate}
\usepackage{hyperref}
\usepackage{bm}
\usepackage[backend=biber, maxnames=99, maxalphanames=99]{biblatex}
\renewbibmacro{in:}{}
\usepackage{mathtools}

\allowdisplaybreaks[3]

\numberwithin{equation}{section}
\allowdisplaybreaks[3]
\makeatother
\newtheorem{thm}{Theorem}[section]

\newtheorem{cor}[thm]{Corollary}

\newtheorem{lem}[thm]{Lemma}

\theoremstyle{definition}

\numberwithin{equation}{section}

\theoremstyle{definition}
\newtheorem*{ack}{Acknowledgements}

\newcommand{\Q}{\mathbb{Q}}

\newcommand{\R}{\mathbb{R}}
\newcommand{\C}{\mathbb{C}}

\newcommand{\Z}{\mathbb{Z}}

\newcommand{\eps}{\varepsilon}

\newcommand{\si}{\sigma}

\newcommand{\re}{\textup{Re}}
\newcommand{\im}{\textup{Im}}

\newcommand{\LL}{\mathcal{L}}
\newcommand{\abs}[1]{\left\lvert#1\right\rvert}

\newcommand{\vv}{\mathbf{v}}
\newcommand{\Log}[1]{\operatorname{Log}(#1)}
\newcommand{\CycUnitGp}{\mathcal{C}}
\newcommand{\Lconst}{C_L}
\newcommand{\prob}[1]{\operatorname{Pr}\left[ #1\right]}
\newcommand{\B}{\mathbf{B}}
\newcommand{\bb}{\mathbf{b}}

\begin{document}
\title[RSG via negative moments of Dirichlet $L$-functions]{Recovering short generators via negative moments of Dirichlet $L$-functions} 
\author[I.-I. Ng, Y. Toma]{Iu-Iong Ng and Yuichiro Toma}
\address[Ng]{Graduate School of Mathematics, Nagoya University, Furo-cho, Chikusa-ku, Nagoya 464-8602, Japan.}
\email{m20048d@math.nagoya-u.ac.jp}
\address[Toma]{Graduate School of Mathematics, Nagoya University, Furo-cho, Chikusa-ku, Nagoya 464-8602, Japan.}
\email{m20034y@math.nagoya-u.ac.jp}
\makeatletter
\@namedef{subjclassname@2020}{\textup{2020} Mathematics Subject Classification}
\makeatother
\maketitle

\begin{abstract}
Cramer, Ducas, Peikert and, Regev~[EUROCRYPT 2016] proposed an efficient algorithm for recovering short generators of principal ideals in $q$-th cyclotomic fields with $q$ being a prime power.
In this paper, we improve their analysis of the dual basis of the log-cyclotomic-unit lattice under the Generalised Riemann Hypothesis and in the case that $q$ is a prime number by the negative square moment of Dirichlet $L$-functions at $s=1$.
As an implication, we obtain a better lower bound on the success probability for the algorithm in this special case.
In order to prove our main result, we also give an analysis of the behaviour of negative $2k$-th moments of Dirichlet $L$-functions at $s=1$.
\end{abstract} 

\section{Introduction}

In order to protect the security from the threats make by the developing fault-tolerant quantum computers, it is important to prepare and move to cryptosystems that resist quantum computing, and that is the purpose of post-quantum cryptography.
From the point of view of attacks, we provide an improvement on the analysis of the algorithm solving a certain lattice problem in a special case.

One of the assumptions we made in the special case is the Generalised Riemann Hypothesis (GRH). It asserts that all nontrivial zeros of Dirichlet $L$-functions lie on the certain line that the real part is $1/2$ (cf.~\cite[\S 20]{D}).

\subsection{Lattice-based cryptosystems}

In 2016, NIST (The National Institute of Standards and Technology, US) announced the Post-Quantum Cryptography Standardization program, in which several candidate schemes are lattice-based, and most of them utilise ideal lattices as their main object.
It results in that most lattice-based cryptosystems rely on the hardness of the algebraic lattice problems such as the shortest vector problem (SVP). 
Hence it is crucial to consider the efficiency of attacks against these lattice problems.
Although there are efficient quantum algorithms for solving some of these lattice problem in special cases with some restrictions or relaxations, lattice-based cryptosystems are believed to be secure against quantum computers.

In this article, we consider the problem called Short Generator of a Principal Ideal Problem (SPIP), which aims to find a sufficiently short generator given a $\Z$-basis of an ideal lattice.
The standard way of attacks toward SPIP consists of two parts (c.f.~\cite{CDPR15}).
\begin{enumerate}
    \item The Principal Ideal Problem (PIP): Given a $\Z$-basis of the principal ideal, find a generator for the ideal.
    \item Short Generator Problem (SGP): Given a generator of a principal ideal, recover a sufficiently short generator for the ideal.
\end{enumerate}
The best known classical algorithms solving PIP, which are proposed by Biasse~\cite{Biasse14PIP} and Fieker~\cite{BF14PIP}, have subexponential running time.
In contrast to classical algorithms, there exists a polynomial time quantum algorithm for solving PIP proposed by Biasse and Song~\cite{biasse2016efficient}. 

The efficiency of the classical SGP algorithm for $q$-th cyclotomic number field with $q$ being a prime power is proved by Cramer, Ducas, Peikert and Regev~\cite{CDPR15}.
Holzer, Wunderer and Buchmann~\cite{HWB17} generalised the result to $q=p_1^{\alpha}p_2^{\beta}$, a product of two odd prime powers.
Okumura, Sugiyama, Yasuda and Takagi~\cite{OSYT18} improved the analysis of the norm of dual basis by calculating the explicit lower and upper bounds of Dirichlet $L$-functions.

We remark an indicative application of SPIP that, based on the quantum PIP algorithm~\cite{biasse2016efficient} and the SGP algorithm~\cite{CDPR15}, there is a polynomial time quantum algorithm approximately solving a special case of SVP, i.e., SVP in ideal lattices of cyclotomic fields, proposed by Cramer, Ducas, and Wesolowski~\cite{CDW21}.

\subsection{Our contributions}

In the present paper, we focus on the SGP algorithm of~\cite{CDPR15} in a special case and show that the analysis of the norm of dual basis can be improved under certain assumptions.
In comparison with~\cite{CDPR15}, where the negative square moment of Dirichlet $L$-functions is bounded with the well-known bound on $L$-functions, we directly calculate the negative square moment.

There are two main contributions presented in this paper.
First, under the GRH, we not only improve the upper bound but also give the asymptotic behaviour of the dual basis of the log-cyclotomic-unit lattice when $q$ is a prime number.
With this asymptotic norm, we show that the success probability of the SGP algorithm has a better lower bound in this special case.
More precisely, when the prime $q$ is large enough, we improve the lower bound on the success probability from $1-(q-3)e^{-t/2}$ with $t=\Theta (\sqrt{q}/\log\log q)$ to $1-(q-3)e^{-t/2}$ with $t=\Theta (\sqrt{q})$.

As a result, we also present the employment of the negative square moment of Dirichlet $L$-functions in cyclotomic fields.
Our second contribution is the very first rigorous analysis of the negative moments of Dirichlet $L$-functions, by which we are able to prove our main result on the dual basis.

The rest of this paper is structured as follows.
A brief background on lattices, number theory, and the prior works is recalled in Section~\ref{sec:pre}.
In Section~\ref{sec:main}, we state the main theorem, which implies an application to the SGP algorithm in Section~\ref{sec:prob}.
In Section~\ref{sec:negative}, we analyse negative moments of $L$-functions for the preparation of proving the main theorem, whose proof is given in Section~\ref{sec:pf}.

\section{Preliminaries}
\label{sec:pre}

We denote by $\langle\cdot ,\cdot\rangle$ the standard inner product over $\C^n$ (or $\R^n$), and by $\lVert\cdot\rVert$ the Euclidean norm.

\subsection{Ideal lattices and lattice problems}

Let $K$ be a number field, i.e., a field extension over $\Q$ with finite degree $[K:\Q]$.
The ring of integers $\mathcal{O}_K$ of $K$ consists of the algebraic integers in the field $K$, and its invertible elements $\mathcal{O}_K^*$ form a group called the unit group.

In this paper, we consider the cyclotomic number fields, i.e., a field extension $K=\Q (\zeta)$ over $\Q$ for some root of unity $\zeta$.
Then the ring of integers has the form $\mathcal{O}_K=\Z [\zeta]$.
$K$ is said to be the $q$-th cyclotomic field if $\zeta$ is the primitive $q$-th root of unit, and then the degree of $K$ is $[K:\Q]=\varphi (q)$, where $\varphi (\cdot)$ is the Euler totient function.

Let $G=\Z_q^*/\{\pm 1\}$ be a group identified with some set of canonical representatives in $\Z_q^*$.
The logarithm embedding of $K$ is defined as
\[\operatorname{Log}: K\rightarrow\R^{\varphi (q)/2},\quad\Log{a}=(\log\abs{\sigma_j(a)})_{j\in G}\quad\text{for all }a\in K,\]
where $\sigma_j$ are the complex embeddings of $K$.

\paragraph{\bf Ideal lattices}
By Dirichlet's unit theorem, $\Log{\mathcal{O}_K^*}$ is a lattice in $\R^{\varphi (q)/2}$ with rank $\varphi (q)/2-1$, i.e., its basis consists of $\varphi (q)/2-1$ linearly independent vectors in $\R^{\varphi (q)/2}$.
The lattices formed by the fractional ideals of (subrings of) the ring of integers $\mathcal{O}_K$ are called the ideal lattices.
In practice, since it is more efficient than dealing with general lattices, the algorithms for lattice-based cryptosystems usually deal with ideal lattices.

Listed below are some lattice problems that are related to our result.

\paragraph{\bf Bounded-distance decoding (BDD)}

Given a lattice basis $\B\subset\R^n$ of the lattice $\LL =\LL (\B)$ and a target point $\mathbf{t}\in\operatorname{span}(\LL)$ with the guarantee that $\min_{\mathbf{v}\in\LL}\lVert\mathbf{v}-\mathbf{t}\rVert\leq r$ for some known $r<\lambda_1(\LL)/2$, where $\lambda_1(\LL)$ denotes the length of the shortest vector in the lattice, find the unique $\mathbf{v}\in\LL$ such that $\lVert\mathbf{v}-\mathbf{t}\rVert\leq r$ (i.e., closest to $\mathbf{t}$).

There is an efficient algorithm solving BDD proposed by Babai~\cite{Bab86}.
We recall that the dual basis $\B^{\vee}=\{\bb_1^{\vee},\dots ,\bb_k^{\vee}\}$ of the lattice $\LL (\B)$ is defined as $\langle\bb_j^{\vee},\bb_{j^{\prime}}\rangle =\delta_{j, j^{\prime}}$.

\begin{thm}[{\cite[Claim 2.1]{CDPR15}}]
\label{thm:Babai}
    Let $\LL\subset\R^n$ be a lattice with a basis $\B$, and let $\mathbf{t}=\mathbf{v}+\mathbf{e}\in\R^n$ for some $\mathbf{v}\in\LL$, $\mathbf{e}\in\R^n$.
    If $\langle\bb_j^{\vee},\mathbf{e}\rangle\in [-\frac{1}{2},\frac{1}{2})$ for all $j$, then on input $\mathbf{t}$ and basis $\B$, Babai's round-off algorithm outputs $\mathbf{v}$.
\end{thm}

\paragraph{\bf Shortest generator problem (SGP)}

Given a generator $h$ of the principal ideal $h\mathcal{O}_K$, recover a generator $g$ of it, i.e., $h\mathcal{O}_K=g\mathcal{O}_K$, with sufficiently short Euclidean norm on its logarithmic embedding.

\subsection{Cyclotomic units and Dirichlet $L$-functions}

In this subsection, we recall the definition of cyclotomic units and Dirichlet $L$-functions, and the relation between them, as an essential observation, is introduced in~\cite{CDPR15}.

\paragraph{\bf Cyclotomic units}
Here we only consider the cyclotomic fields $\Q (\zeta_q)$, where $q$ is a prime power.
For general cyclotomic units, see, for example,~\cite{Was97}.
Define
\[b_j=\frac{\zeta_q^j-1}{\zeta_q-1}.\]
The subgroup $\CycUnitGp =\langle -1,\zeta_q, b_j\mid j\in G\setminus\{ 1\}\rangle <\mathcal{O}_K^*$ of the multiplicative group is called the cyclotomic unit group, and the elements in $\CycUnitGp$ are called the cyclotomic units (c.f.~\cite{Was97}).
It follows that the logarithm embedding $\Log{\CycUnitGp}$ is a sublattice of the log-unit lattice $\Log{\mathcal{O}_K^*}$.
As in~\cite{CDPR15}, we assume that the index $[\Log{\mathcal{O}_K^*}:\Log{\CycUnitGp}]$ is small.
It is clear that the vectors defined as
\[\bb_j=\Log{b_j}=\Log{b_{-j}},\quad j\in G\setminus\{ 1\}\]
form a basis of $\Log{\CycUnitGp}$ (c.f.~\cite{CDPR15}).

\paragraph{\bf Dirichlet $L$-function}
Let $q$ be a positive integer. Let $\chi$ be a Dirichlet character modulo $q$. The corresponding Dirichlet $L$-function is defined to be 
\[
L(s,\chi):=\sum_{n=1}^\infty \frac{\chi(n)}{n^s}
\]
for the real part of $s$, $\si>1$, and it can be continued analytically over the whole plane, except for the possible pole at $s=1$. 

See~\cite{CDPR15} for the way to relate the basis vectors $\bb_j$ to Dirichlet $L$-functions using a certain $G$-circulant matrix and Dirichlet characters.

\subsection{Prior works}

In this subsection, we recall the results on the SGP algorithm in cyclotomic number fields from~\cite{CDPR15}.
Note that we use the notation $q$ for the $q$-th cyclotomic field.

\begin{thm}[{\cite[Theorem 4.1]{CDPR15}}]
\label{thm:313algo}
    Let $D$ be a distribution over $\Q(\zeta)$ with the property that for any tuple of vectors $\vv_1,\dots ,\vv_{\varphi (q)/2-1}\in\R^{\varphi (q)/2}$ of Euclidean norm 1 that are orthogonal to the all-1 vector $\boldsymbol{1}$, the probability that $\abs{\langle\Log{g},\mathbf{v}_i\rangle}<c\sqrt{q}\cdot (\log q)^{-3/2}$ holds for all $i$ is at least some $\alpha >0$, where $g$ is chosen from $D$ and $c$ is a universal constant.
    Then there is an efficient algorithm that given $g^\prime =g\cdot u$, where $g$ is chosen from $D$ and $u\in\CycUnitGp$ is a cyclotomic unit, outputs an element of the form $\pm\zeta^jg$ with probability at least $\alpha$.
\end{thm}

Remember that $g^\prime\in\Q (\zeta)$ is a generator of a principal (fractional) ideal so that $g^\prime\mathcal{O}_K=gu\mathcal{O}_K=g\mathcal{O}_K$.

To show the existence of the polynomial time algorithm, Reference~\cite{CDPR15} gave an upper bound on the length of the generators with the relation to $L(1,\chi)$, considered the candidate distributions, and applied Babai's algorithm~\cite{Bab86} for solving BDD.
The following theorem is a special case of the result given by~\cite{CDPR15} on the upper bound of $\lVert\bb_j^\vee\rVert$ when $q$ is a prime number.

\begin{thm}[{\cite[Theorem 3.1]{CDPR15}}]
\label{thm:313b}
    Let $q$ be a prime number. Then all $\lVert\bb_j^\vee\rVert$ are equal, and $\lVert\bb_j^\vee\rVert^2\leq 2\abs{G}^{-1}\cdot\left(\ell (q)^2+O(1)\right)$, where $\ell(q)$ satisfies that $L(1,\chi) \geq 1/\ell(q)$.
\end{thm}

The role of $L$-functions in the proof is to obtain bounds on $L(1,\chi)^{-1}$ from the well-known bounds $1/\ell (q)\leq L(1,\chi)\leq\ell (q)$.
The existence of such a distribution $D$ in Theorem~\ref{thm:313algo} is shown by the tail bounds of Gaussian distributions.

\begin{lem}[{\cite[Lemma 5.4]{CDPR15}}]
\label{lem:313prob}
    Let $X=1,\dots , X_n, X^\prime_n,\dots , X^\prime_n$ be i.i.d. $N(0, r)$ variables for some $r>0$, and let $\hat{X}_i=(X_i^2+X_i^{\prime 2})^{1/2}$.
    Then, for any vectors $\mathbf{a}^{(1)},\dots ,\mathbf{a}^{(\ell)}\in\R^n$ of Euclidean norm 1 that are orthogonal to the all-1 vector, and every $t\geq C_r$ for some universal constant $C_r$, $\prob{\exists j,\abs{\sum_ia_i^{(j)}\log (\hat{X}_i)}\geq t}\leq 2\ell\exp{(-t/2)}$.
\end{lem}

\section{Asymptotic norm of dual basis vectors}
\label{sec:main}

In this section, we present the statement of our main result on the norm of dual basis $\bb_j^{\vee}$ of the log-cyclotomic-unit lattice $\Log{\CycUnitGp}$, whose proof will be given in Section~\ref{sec:pf}.
More precisely, by investigating the following negative square moment of Dirichlet $L$-function at $s=1$, namely,
\begin{align*}
    \sum_{\substack{\chi \neq \chi_0 \\ \chi(-1)=1}} \frac{1}{f_\chi \abs{L(1,\chi)}^2},
\end{align*}
where $f_\chi$ is the conductor of $\chi$, we obtain the asymptotic behavior of $\lVert \bb_j^\vee \rVert^2$ when $q$ is a prime number.

\begin{thm}
\label{thm:b}
Let $q$ be a prime number. Then all $\lVert \bb_j^\vee \rVert$ are equal, and moreover,
under the generalized Riemann Hypothesis (GRH), we have
\begin{align*}
    \lVert \bb_j^\vee \rVert^2 = \frac{4\zeta(2)}{\zeta(4)} \frac{1}{q}\left(1+O\left( q^{-1+\eps} \right)\right) \quad \text{ as}\quad q \to \infty.
\end{align*}
\end{thm}

Our Theorem~\ref{thm:b} is comparable with the prior work and gives an improvement on the estimation. 
In the case where $q$ is an odd prime, combining Theorem~\ref{thm:313b} and the estimation of $\Lconst$ from~\cite{LLS} and~\cite{LT22}, under the GRH, the result of~\cite{CDPR15} implies only
\[
\lVert \bb_j^\vee \rVert^2 \leq 4\Lconst^2 \frac{(\log\log q)^2}{q}(1+o(1)),  
\]
where $\Lconst =\frac{12e^\gamma}{\pi^2}$ and $\gamma$ is the Euler–Mascheroni constant. Notice that under the GRH, the function $\ell (q)$ in Theorem~\ref{thm:313b}
satisfies that $\ell (q)=\Lconst\log\log q$.

Applications of square moments of Dirichlet $L$-functions to cyclotomic fields are already known. Let $h_q$ be the class number of the cyclotomic field $\Q (\zeta_q)$, where $q$ is an odd prime number. Then it is well-known that
\[
h_q R_q = 2q^{\frac{q}{2}} (2\pi)^{-\frac{q-1}{2}} \prod_{\substack{\chi \bmod q \\ \chi \neq \chi_0}} \abs{L(1,\chi)},
\]
where $R_q$ is the regulator of $\Q (\zeta_q)$. From the earliest work on the square moment from Paley, Selberg (see \cite{Sel46}) and Ankeny and Chowla~\cite{AC51}, several papers giving estimates on the class number of a cyclotomic field, by this direction (cf. \cite{Wa82},\cite{S85,S86},\cite{Z90-2,Z90-1},\cite{Lo93}, \cite{KM94}). 

However, in this paper, we show that the negative square moment of Dirichlet $L$-functions are also applicable to cyclotomic fields, especially, the dual basis of the log-cyclotomic-unit lattice.

\section{Higher success probability for prime $q$}
\label{sec:prob}

In this section, we describe an implication obtained from our main result, Theorem~\ref{thm:b}.
Here we identify the elements of $K=\Q(\zeta)$ by their real and imaginary parts under complex embeddings, i.e., $(\re (\sigma_j(a)),\im (\sigma_j(a)))_{j\in G}$.

We denote by $D(t,\alpha)$ the distribution over $\Q(\zeta)$ with the property that for any tuple of vectors $\vv_1,\dots ,\vv_{\ell}\in\R^n$ of Euclidean norm 1 that are orthogonal to the all-1 vector $\boldsymbol{1}$, the probability that $\abs{\langle\Log{g},\mathbf{v}_i\rangle}<t$ holds for all $i=1,\dots ,\ell$ is at least some $\alpha >0$, where $g$ is chosen from $D(t,\alpha)$ and the parameter $t$ is positive.
Below, we specify the parameters for our setting of $\Q (\zeta_q)$.

\begin{cor}
    \label{cor:prob}
    Let $n=\varphi (q)/2$ and $\ell =n-1$.
    Then $D(t,\alpha)$ with $t=\frac{1}{4\sqrt{2}}\sqrt{\frac{\zeta (4)}{\zeta (2)}q}$ and $\alpha =1-(q-3)e^{-t/2}$ exists for Gaussian distributions that have variance $r$ if $t\geq C_r$.
\end{cor}

Recall that $C_r$ is the universal constant in Lemma~\ref{lem:313prob}.
The corollary follows immediately from Lemma~\ref{lem:313prob} with the observation that $\varphi (q)=q-1$.
We denote such a distribution satisfying Corollary~\ref{cor:prob} by $D_q$.
Then we are ready to give a lower bound on the success probability of the SGP algorithm, which results in a overwhelming probability.

\begin{cor}
\label{cor:algo}
    Let $q$ be an odd prime number.
    Under the GRH, there exists an efficient algorithm that given $g^\prime =g\cdot u$, where $g$ is chosen from $D_q$ and $u\in\CycUnitGp$ is a cyclotomic unit, outputs an element of the form $\pm\zeta^jg$ with probability at least $1-(q-3)e^{-t/2}$, where $t=\frac{1}{4\sqrt{2}}\sqrt{\frac{\zeta (4)}{\zeta (2)}q}$ if $t\geq C_r$.
\end{cor}

We will mainly follow the proof of~\cite[Theorem 4.1]{CDPR15}, and take care of the parameters and probability.

\begin{proof}[Proof of Corollary~\ref{cor:algo}]
The idea is to find the magnitude of $u$ by computing $\Log{u}$, and to divide $g^{\prime}$ by $u$.
Notice that under the logarithmic embedding, we have the equation $\Log{g^{\prime}}=\Log{g}+\Log{u}$, where $\Log{u}\in\Log{\CycUnitGp}$ and $\Log{g}\in\R^{\varphi (q)/2}$, whose form is suitable for BDD.

Then the algorithm goes by first finding $\Log{u}$ with Babai's round-off algorithm (Theorem~\ref{thm:Babai}), then computing $u^{\prime}=\prod b_j^{a_j}$, where $\Log{u}=\sum a_j\bb_j$, and finally outputting $g^{\prime}/u^{\prime}$.
Since $\Log{u^{\prime}}=\Log{u}$ implies that $g^{\prime}/u^{\prime}=\pm\zeta^jg$ for some sign and some $j\in\{1,\dots , q\}$, it suffices to show the fitness for applying Babai's algorithm, and the probability for allowing to apply it.

According to Theorem~\ref{thm:Babai}, in order to apply Babai's round-off algorithm, we need to ensure that $\abs{\langle\Log{g},\bb_j^{\vee}\rangle}<1/2$ for all $j\in G\setminus\{ 1\}$.
By Theorem~\ref{thm:b} and the property of $D_q$, it follows that the overlap
\begin{align*}
    \abs{\langle\Log{g},\bb_j^{\vee}\rangle} & =\lVert\bb_j^{\vee}\rVert\cdot\abs{\left\langle\Log{g},\frac{\bb_j^{\vee}}{\lVert\bb_j^{\vee}\rVert}\right\rangle}\\
    & <\sqrt{\frac{4\zeta(2)}{\zeta(4)} \frac{1}{q}\left(1+O\left( q^{-1+\eps} \right)\right)}\cdot\frac{1}{4\sqrt{2}}\sqrt{\frac{\zeta (4)}{\zeta (2)}q}\\
    & <\frac{1}{2}
\end{align*}
for all $j\in G\setminus\{ 1\}$ meets the requirements of applying Babai's BDD algorithm.
Then by Corollary~\ref{cor:prob}, assuming that $t\geq C_r$, the success probability is lower bounded by $\alpha =1-(q-3)e^{-t/2}$, where $t=\frac{1}{4\sqrt{2}}\sqrt{\frac{\zeta (4)}{\zeta (2)}q}$, as claimed.

\end{proof}

We note that when $q$ is small, according to~\cite[Theorem 5.5]{CDPR15}, the algorithm successes with constant probability.
For $q$ being large enough such that $t\geq C_r$, the dual basis vectors are bounded by $\lVert \bb_j^\vee \rVert\leq 2\Lconst\frac{\log\log q}{\sqrt{q}}\sqrt{(1+o(1))}$ under the GRH in~\cite{CDPR15}, which results in the lower bound on the success probability becoming $1-(q-3)e^{-t/2}$ with $t\geq\frac{1}{4\sqrt{2}\Lconst}\frac{\sqrt{q}}{\log\log q}$ in the case of $q$ being a prime number.
Compared to it, our theoretical bound is asymptotically better.

\section{Negative moments of $L$-functions}
\label{sec:negative}
In this section, let $q$ be a positive integer (we do not require it a prime number). Zhang~\cite{Z93} studied the following $2k$-th negative moments of Dirichlet $L$-functions
\[
\sum_{q \leq Q} \frac{A(q)^k}{\varphi(q)}\sum_{\chi \bmod q} \frac{1}{\abs{L(1,\chi)}^{2k}},
\]
where $A(q)=\prod_{p\mid q}\left(1+p^{-2}\right)$.
Later, Zhang and Deng~\cite{ZD02} also considered a similar sum involving the generalized Gauss sums. The following proof of Theorem \ref{thm:negative moment} has been inspired by the method of~\cite{Z93} and~\cite{ZD02}.

We consider the negative power of $\abs{L(1,\chi)}$ summed only over the non-principal even characters, and prove the following.
\begin{thm} 
    \label{thm:negative moment}
    Let $\chi$ be a Dirichlet character modulo $q$ and let $k$ be a positive integer. Under the GRH, we have
    \begin{align*}
    \sum_{\substack{\chi \neq \chi_0 \\ \chi(-1)=1}} \frac{1}{\abs{L(1,\chi)}^{2k}}&= \frac{C(k)}{2}\varphi(q)\prod_{p\mid q} \left( 1+\frac{\binom{k}{1}^2}{p^2}+\dots+\frac{\binom{k}{k}^2}{p^{2k}} \right)^{-1}\left(1+O \left( q^{-1+\eps}\right)\right),
    \end{align*}
    where 
    \begin{align}
    \label{C(k)}
    C(k) = \prod_{p} \left( 1+\frac{\binom{k}{1}^2}{p^2}+\dots+\frac{\binom{k}{k}^2}{p^{2k}}\right).
    \end{align}
\end{thm}

In order to obtain Theorem \ref{thm:negative moment}, we prove the following Lemmas.
\begin{lem}
\label{lem:average sums of r_k mu}
    Let $(a,q)=1$ and 
    \[
    r_k(n) = \sum_{n_1n_2\dots n_k=n} \mu(n_1)\mu(n_2)\dots \mu(n_k).
    \]
    Then for given $\eps > 0$ and $q \leq x$, under the GRH we have
    \[
    \sum_{\substack{n \leq x \\ n \equiv a \pmod q}} r_k(n) \ll x^{\frac{1}{2}+\eps}.
    \]
\end{lem}
\begin{proof}
This can be proved in the same manner as the case of M\"obius sums in arithmetic progression. Since $(a,q)=1$, the left hand side can be rewritten as
\[
\sum_{\substack{n \leq x \\ n \equiv a \pmod q}} r_k(n) = \frac{1}{\varphi(q)} \sum_{\chi} \overline{\chi}(a) \sum_{n \leq x} r_k(n)\chi(n).
\]
By Perron's formula~\cite[Theorem 5.2 and Corollary 5.3]{MV}, we have
\begin{align}
\label{Perron formula}
    \sideset{}{'}{\sum}_{n \leq x} r_k(n)\chi(n) = \frac{1}{2\pi i} \int_{1+\frac{1}{\log x}-ix}^{1+\frac{1}{\log x}+ix} \frac{x^s}{L(s,\chi)^k s}ds+R,
\end{align}
with 
\begin{align*}
    R \ll \sum_{\substack{x/2\leq n\leq 2x \\ n \neq x}} \abs{r_k(n)} \min \left\{1, \frac{1}{\abs{x-n}}\right\}+\sum_{n=1}^\infty \frac{\abs{r_k(n)}}{n^{1+\frac{1}{\log x}}}.
\end{align*}
Here, $\sum^\prime$ indicates that if $x$ is an integer, then the last term is to be halved. 

Since $\abs{r_k(n)} \leq d_k(n) \leq n^\eps$, where $d_k(n)$ indicates the number of ways to express $n$ as a product of $k$ factors, we have $R \ll x^\eps$. Under the GRH, the residue theorem implies that the right hand side in (\ref{Perron formula}) is equal to
\begin{align*}
    \frac{1}{2\pi i}\left( \int_{\frac{1}{2}+\frac{1}{\log\log qx}+ix}^{1+\frac{1}{\log x}+ix}+ \int_{\frac{1}{2}+\frac{1}{\log\log qx}-ix}^{\frac{1}{2}+\frac{1}{\log\log qx}+ix}-\int_{\frac{1}{2}+\frac{1}{\log\log qx}-ix}^{1+\frac{1}{\log x}-ix}\right) \frac{x^s}{L(s,\chi)^k s}ds.
\end{align*}
If $\chi$ is imprimitive, then 
\[
L(s,\chi) = \prod_{p \mid q} \left(1-\frac{\chi^*(p)}{p^s} \right) L(s,\chi^*),
\]
where $\chi^*$ is the primitive Dirichlet character modulo $d$ which induces $\chi$. We note that for $\si\geq 1/2$ it holds that
\begin{align*}
    \prod_{p \mid q} \abs{1-\frac{\chi^*(p)}{p^s}}^{-1} \leq \prod_{p \mid q} \left(1+\frac{1}{p^\frac{1}{2}-1}\right) \leq \exp\left( \sum_{p \mid q}\frac{1}{p^\frac{1}{2}-1}\right)\leq \exp\left( \omega(q)\right),
\end{align*}
here $\omega(n)$ denotes the number of distinct prime divisor of $n$.
Then combining the conditional estimate 
\[
\abs{L(s,\chi^*)}^{-1} \leq \exp \left( C\frac{\log (d(\abs{t}+1))}{\log\log (d(\abs{t}+1))} \right) 
\]
for $1/2+(\log\log (d(\abs{t}+1)))^{-1}\leq \si \leq 3/2$ (cf. \cite[\S 13.2 and exercises]{MV}) and for some constant $C>0$, and $\omega(q) \ll \log q /(\log\log q)$ (cf. \cite[Theorem 2.10]{MV}) with replacing the constants, we obtain
\begin{align*}
\abs{L(s,\chi)}^{-k} &= \prod_{p \mid q} \abs{1-\frac{\chi^*(p)}{p^s}}^{-k} \abs{L(s,\chi^*)}^{-k} \\
&\leq \exp \left( c_k \left(\frac{\log q}{\log\log q}+\frac{\log (q(\abs{t}+1))}{\log\log (q(\abs{t}+1))} \right)\right)
\end{align*}
for $1/2+(\log\log (q(\abs{t}+1)))^{-1}\leq \si \leq 3/2$. By the above inequality with replacing the constants again, we have 
\begin{align*}
    \int_{\frac{1}{2}+\frac{1}{\log\log qx}\pm ix}^{1+\frac{1}{\log x}\pm ix} \frac{x^s}{L(s,\chi)^k s}ds & \ll \exp \left( c_k \frac{\log x}{\log\log x} \right) \ll x^\eps, 
\end{align*}
and 
\begin{align*}
    \int_{\frac{1}{2}+\frac{1}{\log\log qx}-ix}^{\frac{1}{2}+\frac{1}{\log\log qx}+ix} \frac{x^s}{L(s,\chi)^k s}ds &\ll x^\frac{1}{2}(\log x) \exp \left( c_k\frac{\log x}{\log\log x}\right) \ll x^{\frac{1}{2}+\eps}
\end{align*}
for $q\leq x$. 
\end{proof}

\begin{lem}[Character orthogonality]
\label{lem:Character orthogonality}
     For $n_1, n_2$ integers coprime to $q$, we have
     \[
     \sum_{\substack{\chi \\ \chi(-1)=(-1)^{\mathfrak{a}}}} \chi(n_1)\overline{\chi}(n_2) = \frac{\varphi(q)}{2}\bm{1}_{n_1 \equiv n_2 \ (q)}+(-1)^{\mathfrak{a}} \frac{\varphi(q)}{2}\bm{1}_{n_1 \equiv -n_2 \ (q)}.
     \]
     The sum on the left-hand side vanishes if $(n_1n_2, q)\neq 1$.
\end{lem}
\begin{proof}
    The proof is standard. For $n_1, n_2$ integers coprime to $q$, from
    \[
    \sum_{\chi} \chi(n_1)\overline{\chi}(n_2) =\begin{cases}
        \varphi(q) & \text{ if } n_1 \equiv n_2 \pmod q, \\
        0 & \text{ otherwise }, \\
        \end{cases}
    \]
    and replacing $n_2$ with $-n_2$, we get the desired result.
    \end{proof}

\begin{proof}[Proof of Theorem~\ref{thm:negative moment}]
Let $q \geq 2$. For $\si>1$, $1/L(s,\chi)^k$ has the Dirichlet series expression
\begin{align}
\label{Dirichlet series}
    \frac{1}{L(s,\chi)^k}=\sum_{m=1}^\infty \frac{r_k(m)\chi(m)}{m^s}.
\end{align}
By combining (\ref{Dirichlet series}) and \cite[(4.22) and (6.19)]{MV}, we find for the principal character $\chi_0$ that
\[
\frac{1}{L(1,\chi_0)^k}=\prod_{p \mid q} \left(1-\frac{1}{p}\right)^{-k} \left(\sum_{m=1}^\infty \frac{\mu(m)}{m}\right)^k =0.
\]
Hence it suffices to consider the moment
\[
\sum_{\substack{\chi \\ \chi(-1)=1}} \frac{1}{\abs{L(1,\chi)}^{2k}},
\]
where $\chi$ runs over all even Dirichlet character modulo $q$. 

For a nonprincipal even character $\chi$, (\ref{Dirichlet series}) holds at $s=1$. By partial summation with $y = q^2$ we have
\begin{align*}
    \frac{1}{L(1,\chi)^k} &= \sum_{1\leq m \leq y}\frac{\chi(m)r_k(m)}{m}+\int_y^\infty\frac{A(t,r_k,\chi)}{t^2} dt,
\end{align*}
where
\[
A(t,r_k,\chi)= \sum_{y < m \leq t} \chi(m)r_k(m).
\]

Hence
\begin{align*}
    \sum_{\substack{\chi \\ \chi(-1)=1}} \frac{1}{\abs{L(1,\chi)}^{2k}} &= \sum_{\substack{\chi \\ \chi(-1)=1}} \left[\left( \sum_{m \leq y} \frac{\chi(m)r_k(m)}{m}+\int_y^\infty \frac{A(t,r_k,\chi)}{t^2}dt \right)\right. \\
    &\qquad\qquad \times \left.\left( \sum_{n \leq y} \frac{\overline{\chi}(n)r_k(n)}{n}+\int_y^\infty \frac{A(u,r_k,\overline{\chi})}{u^2}du \right)\right] \\
    &= \sum_{\substack{\chi \\ \chi(-1)=1}} \sum_{m,n \leq y} \frac{\chi(m)\overline{\chi}(n)r_k(m)r_k(n)}{mn} \\
    &\qquad\qquad + \sum_{\substack{\chi \\ \chi(-1)=1}} \sum_{m \leq y} \frac{\chi(m)r_k(m)}{m} \int_y^\infty \frac{A(u,r_k,\overline{\chi})}{u^2}du \\
    &\qquad\qquad + \sum_{\substack{\chi \\ \chi(-1)=1}} \sum_{n \leq y} \frac{\overline{\chi}(n)r_k(n)}{n} \int_y^\infty \frac{A(t,r_k,\chi)}{t^2}dt \\
    &\qquad\qquad + \sum_{\substack{\chi \\ \chi(-1)=1}} \int_y^\infty \frac{A(t,r_k,\chi)}{t^2}dt \int_y^\infty \frac{A(u,r_k,\overline{\chi})}{u^2}du \\
    &= M_1+M_2+M_3+M_4,
\end{align*}
say. By Lemma~\ref{lem:Character orthogonality}, we have
\begin{align*}
    M_1 &= \sum_{\substack{\chi \\ \chi(-1)=1}} \sum_{m,n \leq y} \frac{\chi(m)\overline{\chi}(n)r_k(m)r_k(n)}{mn} \\
    &=\frac{\varphi(q)}{2} \sum_{\substack{m \leq y \\ (m,q)=1}} \frac{r_k(m)^2}{m^2}+\frac{\varphi(q)}{2} \sum_{\substack{m,n \leq y \\ m \neq n \\ (mn,q)=1}} \frac{r_k(m)r_k(n)}{mn}\left( \bm{1}_{m \equiv n}+\bm{1}_{m \equiv -n } \right) \\
    &=\frac{\varphi(q)}{2} \sum_{\substack{m =1 \\ (m,q)=1}}^\infty \frac{r_k(m)^2}{m^2} + O \left( \varphi(q)q^{-2+\eps}\right) +\frac{\varphi(q)}{2} \sum_{\substack{m,n \leq y \\ m \neq n \\ (mn,q)=1}} \frac{r_k(m)r_k(n)}{mn}\left( \bm{1}_{m \equiv n}+\bm{1}_{m \equiv -n } \right).
\end{align*}

Next, we estimate nondiagonal contribution. 
\begin{align}
\begin{split}
\label{M_1-non diagonal}
    & \frac{\varphi(q)}{2} \sum_{\substack{m,n \leq y \\ m \neq n \\ (mn,q)=1}} \frac{r_k(m)r_k(n)}{mn}\left( \bm{1}_{m \equiv n}+\bm{1}_{m \equiv -n } \right) \\
    &=\varphi(q) \left( \sum_{\substack{m \leq y \\ (m,q)=1}} \frac{r_k(m)}{m} \sum_{\substack{m<n \leq y \\ n \equiv m^\prime \pmod q}} \frac{r_k(n)}{n} + \sum_{\substack{m \leq y \\ (m,q)=1}} \frac{r_k(m)}{m} \sum_{\substack{m<n \leq y \\ 
    n \equiv -m^\prime \pmod q}} \frac{r_k(n)}{n} \right),
\end{split}
\end{align}
where $m^\prime= m-q \lfloor m/q \rfloor$. Here, if $n \equiv \pm m^\prime \pmod q$, then $(n,q)=1$ since $(m,q)=1$. So the restriction $(n,q)=1$ is removed from the inner sum. We denote the sums on the right hand side of (\ref{M_1-non diagonal}) by $M_{11}$ and $M_{12}$, respectively. 

By using the fact $\abs{r_k(n)} \leq d_k(n) \ll n^\eps$ and partial summation we have
\begin{align*}
    M_{11} &\ll \varphi(q) \sum_{\substack{m \leq y \\ (m,q)=1}} \frac{d_k(m)}{m} \sum_{\substack{m<n \leq y \\ n \equiv m^\prime \pmod q}} \frac{d_k(n)}{n} \\
    &\ll \varphi(q)y^\eps \sum_{\substack{m \leq y \\ (m,q)=1}} \frac{d_k(m)}{m} \sum_{\frac{m-m^\prime}{q}<\ell \leq \frac{y-m^\prime}{q}} \frac{1}{q\ell+m^\prime} \\
    &\ll \varphi(q)y^\eps \sum_{\substack{m \leq y \\ (m,q)=1}} \frac{d_k(m)}{m} \frac{1}{q}\int_{\frac{m-m^\prime}{q}}^{\frac{y-m^\prime}{q}} \frac{d \xi}{\xi+\frac{m^\prime}{q}} \\
    &\ll \varphi(q) q^{-1+\eps}.
\end{align*}
By the same argument as $M_{11}$, we have $M_{12} \ll \varphi(q) q^{-1+\eps}$.

Hence we have
\begin{align}
\label{M_1}
    M_1 &= \frac{\varphi(q)}{2}\sum_{\substack{m =1 \\ (m,q)=1}}^\infty \frac{r_k(m)^2}{m^2}
    + O \left( \varphi(q) q^{-1+\eps} \right).
\end{align}

We then estimate $M_{2}$. By the same argument as $M_1$, Lemma~\ref{lem:Character orthogonality} implies that
\begin{align*}
    M_{2} &=\frac{\varphi(q)}{2} \sum_{\substack{m \leq y \\ (m,q)=1}} \frac{r_k(m)}{m} \int_y^\infty u^{-2} \sum_{\substack{y<n \leq u \\ n \equiv m^\prime \pmod q}} r_k(n) du\\
    &\qquad\qquad + \frac{\varphi(q)}{2} \sum_{\substack{m \leq y \\ (m,q)=1}} \frac{r_k(m)}{m} \int_y^\infty u^{-2} \sum_{\substack{y<n \leq u \\ 
    n \equiv -m^\prime \pmod q}} r_k(n) du.
\end{align*}
Under the GRH, by applying Lemma~\ref{lem:average sums of r_k mu}, we have
\begin{align}
\begin{split}
    \label{M_2}
    M_2 &\ll \varphi(q) \sum_{\substack{m \leq y \\ (m,q)=1}} \frac{d_k(m)}{m} \int_y^\infty u^{-\frac{3}{2}+\eps} du \\
    &\ll \varphi(q) y^{-\frac{1}{2}+\eps} \sum_{m \leq y} \frac{d_k(m)}{m} \\
    &\ll \varphi(q) q^{-1+\eps}. 
\end{split}
\end{align}
By the same argument as $M_{2}$, $M_3$ can be 
estimated as
\begin{align}
\label{M_3}
    M_3 &\ll \varphi(q) q^{-1+\eps}.
\end{align}

Finally, we estimate $M_{4}$. By Lemma~\ref{lem:Character orthogonality}, we have
\begin{align*}
    M_4 &= \frac{\varphi(q)}{2}\int_y^\infty t^{-2} \sum_{\substack{y<m\leq t \\ (m,q)=1}} r_k(m) \\
    &\qquad\qquad \times \int_y^\infty u^{-2} \sum_{\substack{y<n\leq u \\ (n,q)=1}} r_k(n)  \left(\bm{1}_{m=n}+\bm{1}_{\substack{m\neq n \\ m\equiv n}}+\bm{1}_{\substack{m\neq n \\ m\equiv -n}} \right)du dt \\
    &= \frac{\varphi(q)}{2} \int_y^\infty t^{-2} \sum_{\substack{y<m\leq t \\ (m,q)=1}} r_k(m) \int_y^\infty u^{-2} \sum_{\substack{y<n\leq u \\ (n,q)=1 \\
    m=n}} r_k(n) dudt \\
    &\qquad\qquad+\varphi(q) \int_y^\infty t^{-2} \sum_{\substack{y<m\leq t \\ (m,q)=1}} r_k(m) \int_y^\infty u^{-2}  \sum_{\substack{m<n\leq u \\ n \equiv m^\prime \pmod q}} r_k(n) du dt \\
    &\qquad\qquad+\varphi(q) \int_y^\infty t^{-2} \sum_{\substack{y<m\leq t \\ (m,q)=1}} r_k(m) \int_y^\infty u^{-2}  \sum_{\substack{m<n\leq u \\ 
    n \equiv -m^\prime \pmod q}} r_k(n) du dt.
\end{align*}
We then estimate the diagonal contribution. 
\begin{align*}
    & \frac{\varphi(q)}{2} \int_y^\infty t^{-2} \sum_{\substack{y<m\leq t \\ (m,q)=1}} r_k(m) \int_y^\infty u^{-2} \sum_{\substack{y<n\leq u \\ (n,q)=1 \\
    m=n}} r_k(n) dudt \\
    &=\frac{\varphi(q)}{2} \int_y^\infty t^{-2} \sum_{\substack{y<m\leq t \\ (m,q)=1}} r_k(m)^2 \int_m^\infty u^{-2} dudt \\
    &\ll \varphi(q) \int_y^\infty t^{-2} \sum_{m\leq t} \frac{r_k(m)^2}{m} dt \\
    &\ll \varphi(q) \int_y^\infty t^{-2+\eps} dt \\
    &\ll \varphi(q) q^{-2+\eps}. 
\end{align*}

Next, we estimate the nondiagonal contribution of $M_4$. Under the GRH, we have
\begin{align*}
    & \varphi(q) \int_y^\infty t^{-2} \sum_{\substack{y<m\leq t \\ (m,q)=1}} r_k(m) \int_y^\infty u^{-2}  \sum_{\substack{m<n\leq u \\ n \equiv m^\prime \pmod q}} r_k(n) du dt \\
    &= \varphi(q) \int_y^\infty t^{-2} \sum_{\substack{y<m\leq t \\ (m,q)=1}} r_k(m) \int_m^\infty u^{-2}  \sum_{\substack{m<n\leq u \\ n \equiv m^\prime \pmod q}} r_k(n) du dt \\
    &\ll \varphi(q) \int_y^\infty t^{-2} \sum_{\substack{y<m\leq t \\ (m,q)=1}} d_k(m) \int_m^\infty u^{-\frac{3}{2}+\eps} du dt \\
    &\ll  \varphi(q) \int_y^\infty t^{-2} \sum_{m\leq t } \frac{d_k(m)}{m^{\frac{1}{2}-\eps}} dt \\
    &\ll \varphi(q) q^{-1+\eps}. 
\end{align*}
By the same argument as above, 
\begin{align*}
    & \varphi(q) \int_y^\infty t^{-2} \sum_{\substack{y<m\leq t \\ (m,q)=1}} r_k(m) \int_y^\infty u^{-2}  \sum_{\substack{m<n\leq u \\ 
    n \equiv -m^\prime \pmod q}} r_k(n) du dt \\
    &\ll \varphi(q) q^{-1+\eps}.
\end{align*}
Thus, we have
\begin{align}
\label{M_4}
    M_4 &\ll \varphi(q) q^{-1+\eps}.
\end{align}

Therefore, combining (\ref{M_1}), (\ref{M_2}), (\ref{M_3}) and (\ref{M_4}) we have
\begin{align*}
    \sum_{\substack{\chi \neq \chi_0 \\ \chi(-1)=1}} \frac{1}{\abs{L(1,\chi)}^{2k}}&= \sum_{\substack{\chi \\ \chi(-1)=1}} \frac{1}{\abs{L(1,\chi)}^{2k}} \\
    &= \frac{\varphi(q)}{2}\sum_{\substack{m =1 \\ (m,q)=1}}^\infty \frac{r_k(m)^2}{m^{2}} 
    + O \left( \varphi(q)q^{-1+\eps} \right).
\end{align*}
We see that $r_k(n)$ is supported on the set of $(k+1)$-free integers and $r_k(p^v)^2= \binom{k}{v}^2$ if $v \leq k$. Hence the coefficient of $\varphi(q)/2$ is 
\begin{align*}
    \sum_{\substack{m =1 \\ (m,q)=1}}^\infty \frac{r_k(m)^2}{m^{2}} &=\prod_{\substack{p \\ p \nmid q}} \left( 1+\frac{\binom{k}{1}^2}{p^2}+\dots+\frac{\binom{k}{k}^2}{p^{2k}}\right) \\
    & =C(k) \prod_{p \mid q} \left( 1+\frac{\binom{k}{1}^2}{p^2}+\dots+\frac{\binom{k}{k}^2}{p^{2k}}\right)^{-1}.
\end{align*}
We complete the proof.
\end{proof}

By the same argument as the proof of Theorem \ref{thm:negative moment}, we immediately obtain the followings.

\begin{thm} 
    Let $\chi$ be a Dirichlet character modulo $q$ and $k$ be a positive integer. Under the GRH, we have
    \begin{align*}
    \sum_{\substack{\chi \\ \chi(-1)=-1}} \frac{1}{\abs{L(1,\chi)}^{2k}}&= \frac{C(k)}{2}\varphi(q)\prod_{p\mid q} \left( 1+\frac{\binom{k}{1}^2}{p^2}+\dots+\frac{\binom{k}{k}^2}{p^{2k}} \right)^{-1}\left(1+O \left( q^{-1+\eps}\right)\right),
    \end{align*}
    where $C(k)$ is given by (\ref{C(k)}).
\end{thm}

\begin{thm} 
    Let $\chi$ be a Dirichlet character modulo $q$ and $k$ be a positive integer. Under the GRH, we have
    \begin{align*}
    \sum_{\chi \neq \chi_0 } \frac{1}{\abs{L(1,\chi)}^{2k}}&= C(k) \varphi(q)\prod_{p\mid q} \left( 1+\frac{\binom{k}{1}^2}{p^2}+\dots+\frac{\binom{k}{k}^2}{p^{2k}} \right)^{-1}\left(1+O \left( q^{-1+\eps}\right)\right),
    \end{align*}
    where $C(k)$ is given by (\ref{C(k)}).
\end{thm}

\section{Proof of Theorem~\ref{thm:b}}
\label{sec:pf}
The first part follows immediately from Theorem \ref{thm:313b}. Moreover, in~\cite{CDPR15}, it is shown that if $q=p^\ell$ is a prime power, it holds that
\begin{align*}
    \lVert \bb_j^\vee \rVert^2 = \frac{4}{\abs{G}} \sum_{\substack{\chi \neq \chi_0 \\ \chi(-1)=1}} \frac{1}{f_\chi \abs{L(1,\chi)}^2}.
\end{align*}
If $q$ is a prime number, then $f_\chi=q$. By applying Theorem~\ref{thm:negative moment} with $k=1$, we have
\begin{align*}
    \lVert \bb_j^\vee \rVert^2 &= \frac{4}{\varphi(q)} \sum_{\substack{\chi \neq \chi_0 \\ \chi(-1)=1}} \frac{1}{q\abs{L(1,\chi)}^2} \\
    &= 4C(1)\frac{q}{q^2+1}\left(1+O\left( q^{-1+\eps} \right)\right) \\
    &= \frac{4C(1)}{q}\left(1+O\left( q^{-1+\eps} \right)\right).
\end{align*}
The coefficient $C(1)$ is 
\begin{align*}
    \prod_{p} \left( 1+\frac{1}{p^2}\right) = \frac{\zeta(2)}{\zeta(4)}
\end{align*}
as claimed.

\begin{ack} 
The authors would like to thank Yuya Kanado for pointing out a mistake. The authors would also like to thank Professor Fran{\c{c}}ois Le Gall and Professor Kohji Matsumoto for giving us valuable comments. I.-I. Ng is supported by MEXT Q-LEAP grant No. JPMXS0120319794.
Y. Toma is supported by Grant-in-Aid for JSPS Research Fellow grant No. 24KJ1235. 
\end{ack} 

\printbibliography[heading=none]

@incollection {CDPR15,
    AUTHOR = {Cramer, Ronald and Ducas, L\'{e}o and Peikert, Chris and
              Regev, Oded},
     TITLE = {Recovering short generators of principal ideals in cyclotomic
              rings},
 BOOKTITLE = {Advances in cryptology---{EUROCRYPT} 2016. {P}art {II}},
    SERIES = {Lecture Notes in Comput. Sci.},
    VOLUME = {9666},
     PAGES = {559--585},
% PUBLISHER = {Springer, Berlin},
      YEAR = {2016},
%       DOI = {10.1007/978-3-662-49896-5\_20},
}

@incollection {HWB17,
    AUTHOR = {Holzer, Patrick and Wunderer, Thomas and Buchmann, Johannes
              A.},
     TITLE = {Recovering short generators of principal fractional ideals in
              cyclotomic fields of conductor {$p^\alpha q^\beta$}},
 BOOKTITLE = {Progress in cryptology---{INDOCRYPT} 2017},
    SERIES = {Lecture Notes in Comput. Sci.},
    VOLUME = {10698},
     PAGES = {346--368},
% PUBLISHER = {Springer, Cham},
      YEAR = {2017},
%       DOI = {10.1007/978-3-319-71667-1\_18},
}

@book {MV,
    AUTHOR = {Montgomery, Hugh L. and Vaughan, Robert C.},
     TITLE = {Multiplicative number theory. {I}. {C}lassical theory},
    SERIES = {Cambridge Studies in Advanced Mathematics},
    VOLUME = {97},
      YEAR = {2007},
 PUBLISHER = {Cambridge University Press, Cambridge},
%     PAGES = {xviii+552},
}

@book{D,
  title={Multiplicative number theory},
  author={Davenport, Harold},
  volume={74},
  year={2013},
  publisher={Springer Science \& Business Media}
}

@article {OSYT18,
    AUTHOR = {Okumura, Shinya and Sugiyama, Shingo and Yasuda, Masaya and
              Takagi, Tsuyoshi},
     TITLE = {Security analysis of cryptosystems using short generators over
              ideal lattices},
   JOURNAL = {Jpn. J. Ind. Appl. Math.},
  FJOURNAL = {Japan Journal of Industrial and Applied Mathematics},
    VOLUME = {35},
      YEAR = {2018},
    NUMBER = {2},
     PAGES = {739--771},
%       DOI = {10.1007/s13160-018-0306-z},
}

@article {LLS,
    AUTHOR = {Lamzouri, Youness and Li, Xiannan and Soundararajan, Kannan},
     TITLE = {Conditional bounds for the least quadratic non-residue and
              related problems},
   JOURNAL = {Math. Comp.},
  FJOURNAL = {Mathematics of Computation},
    VOLUME = {84},
      YEAR = {2015},
    NUMBER = {295},
     PAGES = {2391--2412},
       %DOI = {10.1090/S0025-5718-2015-02925-1},
}

@article {LT22,
    AUTHOR = {Languasco, Alessandro and Trudgian, Timothy S.},
     TITLE = {Uniform effective estimates for {$|L(1,\chi)|$}},
   JOURNAL = {J. Number Theory},
  FJOURNAL = {Journal of Number Theory},
    VOLUME = {236},
      YEAR = {2022},
     PAGES = {245--260},
%       DOI = {10.1016/j.jnt.2021.07.019},
}

@article {Bab86,
    AUTHOR = {Babai, L.},
     TITLE = {On {L}ov\'{a}sz' lattice reduction and the nearest lattice
              point problem},
   JOURNAL = {Combinatorica},
  FJOURNAL = {Combinatorica. An International Journal of the J\'{a}nos
              Bolyai Mathematical Society},
    VOLUME = {6},
      YEAR = {1986},
    NUMBER = {1},
     PAGES = {1--13},
%       DOI = {10.1007/BF02579403},
}

@book {Was97,
    AUTHOR = {Washington, Lawrence C.},
     TITLE = {Introduction to cyclotomic fields},
    SERIES = {Graduate Texts in Mathematics},
    VOLUME = {83},
 PUBLISHER = {Springer-Verlag, New York},
      YEAR = {1982},
%       DOI = {10.1007/978-1-4684-0133-2},
}

@article{AC51,
    AUTHOR = {Ankeny, Nesmith C. and Chowla, Sarvadaman},
     TITLE = {The Class Number of the Cyclotomic Field}, 
   JOURNAL = {Canad J. Math},
  FJOURNAL = {Canadian Journal of Mathematics},
    VOLUME = {3},
      YEAR = {1951},
     PAGES = {486--494},
}

@inproceedings {biasse2016efficient,
    AUTHOR = {Biasse, Jean-Fran\c{c}ois and Song, Fang},
     TITLE = {Efficient quantum algorithms for computing class groups and
              solving the principal ideal problem in arbitrary degree number
              fields},
 BOOKTITLE = {Proceedings of the {T}wenty-{S}eventh {A}nnual {ACM}-{SIAM}
              {S}ymposium on {D}iscrete {A}lgorithms},
     PAGES = {893--902},
 PUBLISHER = {ACM, New York},
      YEAR = {2016},
%       DOI = {10.1137/1.9781611974331.ch64},
}

@article {CDW21,
    AUTHOR = {Cramer, Ronald and Ducas, L\'{e}o and Wesolowski, Benjamin},
     TITLE = {Mildly short vectors in cyclotomic ideal lattices in quantum
              polynomial time},
   JOURNAL = {J. ACM},
  FJOURNAL = {Journal of the ACM},
    VOLUME = {68},
      YEAR = {2021},
    NUMBER = {2},
     PAGES = {Art. 8, 26},
%       DOI = {10.1145/3431725},
}

@article {KM94,
    AUTHOR = {Katsurada, Masanori and Matsumoto, Kohji},
     TITLE = {The mean values of {D}irichlet {$L$}-functions at integer
              points and class numbers of cyclotomic fields},
   JOURNAL = {Nagoya Math. J.},
  FJOURNAL = {Nagoya Mathematical Journal},
    VOLUME = {134},
      YEAR = {1994},
     PAGES = {151--172},
%       DOI = {10.1017/S0027763000004906},
}

@article {Lo93,
    AUTHOR = {Louboutin, St\'{e}phane},
     TITLE = {Quelques formules exactes pour des moyennes de fonctions {$L$}
              de {D}irichlet},
   JOURNAL = {Canad. Math. Bull.},
  FJOURNAL = {Canadian Mathematical Bulletin. Bulletin Canadien de
              Math\'{e}matiques},
    VOLUME = {36},
      YEAR = {1993},
    NUMBER = {2},
     PAGES = {190--196},
%       DOI = {10.4153/CMB-1993-028-8},
}

@article {ZD02,
    AUTHOR = {Zhang, Wenpeng and Deng, Yuping},
     TITLE = {A hybrid mean value of the inversion of {$L$}-functions and
              general quadratic {G}auss sums},
   JOURNAL = {Nagoya Math. J.},
  FJOURNAL = {Nagoya Mathematical Journal},
    VOLUME = {167},
      YEAR = {2002},
     PAGES = {1--15},
%       DOI = {10.1017/S002776300002540X},
}

@article {Z90-1,
    AUTHOR = {Zhang, Wen Peng},
     TITLE = {On the mean value of the {$L$}-function},
   JOURNAL = {J. Math. Res. Exposition},
  FJOURNAL = {Journal of Mathematical Research and Exposition},
    VOLUME = {10},
      YEAR = {1990},
    NUMBER = {3},
     PAGES = {355--360},
}

@article {Z90-2,
    AUTHOR = {Zhang, Wen Peng},
     TITLE = {An elementary result about {$L$}-functions},
   JOURNAL = {Adv. in Math. (China)},
  FJOURNAL = {Advances in Mathematics (China). Shuxue Jinzhan},
    VOLUME = {19},
      YEAR = {1990},
    NUMBER = {4},
     PAGES = {478--487},
}

@article {Wa82,
    AUTHOR = {Walum, Herbert},
     TITLE = {An exact formula for an average of {$L$}-series},
   JOURNAL = {Illinois J. Math.},
  FJOURNAL = {Illinois Journal of Mathematics},
    VOLUME = {26},
      YEAR = {1982},
    NUMBER = {1},
     PAGES = {1--3},
}

@article {BF14PIP,
    AUTHOR = {Biasse, Jean-Fran\c{c}ois and Fieker, Claus},
     TITLE = {Subexponential class group and unit group computation in large
              degree number fields},
   JOURNAL = {LMS J. Comput. Math.},
  FJOURNAL = {LMS Journal of Computation and Mathematics},
    VOLUME = {17},
      YEAR = {2014},
     PAGES = {385--403},
%       DOI = {10.1112/S1461157014000345},
}

@article {Biasse14PIP,
    AUTHOR = {Biasse, Jean-Fran\c{c}ois},
     TITLE = {Subexponential time relations in the class group of large
              degree number fields},
   JOURNAL = {Adv. Math. Commun.},
  FJOURNAL = {Advances in Mathematics of Communications},
    VOLUME = {8},
      YEAR = {2014},
    NUMBER = {4},
     PAGES = {407--425},
%       DOI = {10.3934/amc.2014.8.407},
}

@incollection {S85,
    AUTHOR = {Slavutskii, Ilya Sh.},
     TITLE = {Mean value of {$L$}-functions and the class number of a
              cyclotomic field},
 BOOKTITLE = {Algebraic systems with one action and relation},
     PAGES = {122--129},
 PUBLISHER = {Leningrad. Gos. Ped. Inst., Leningrad},
      YEAR = {1985},
}

@article {S86,
    AUTHOR = {Slavutskii, Ilya Sh.},
     TITLE = {Mean value of {$L$}-functions and the class number of a
              cyclotomic field},
   JOURNAL = {Zap. Nauchn. Sem. Leningrad. Otdel. Mat. Inst. Steklov.
              (LOMI)},
  FJOURNAL = {Zapiski Nauchnykh Seminarov Leningradskogo Otdeleniya
              Matematicheskogo Instituta imeni V. A. Steklova Akademii Nauk
              SSSR (LOMI)},
    VOLUME = {154},
      YEAR = {1986},
     PAGES = {136--143, 178--179},
%       DOI = {10.1007/BF01374991},
}

@article {Sel46,
    AUTHOR = {Selberg, Atle},
     TITLE = {Contributions to the theory of {D}irichlet's {$L$}-functions},
   JOURNAL = {Skr. Norske Vid.-Akad. Oslo I},
  FJOURNAL = {Skrifter Utgitt av det Norske Videnskaps-Akademi i Oslo. I.
              Mat.-Naturv. Klasse},
    VOLUME = {1946},
      YEAR = {1946},
    NUMBER = {3},
     PAGES = {62},
}

@article {Z93,
    AUTHOR = {Zhang, Wenpeng},
     TITLE = {The {$2k$}th power means of inverses of {D}irichlet
              {$L$}-functions},
   JOURNAL = {Chinese Ann. Math. Ser. A},
  FJOURNAL = {Chinese Annals of Mathematics. Series A. Shuxue Niankan. A Ji},
    VOLUME = {14},
      YEAR = {1993},
    NUMBER = {1},
     PAGES = {1--5},
}
\end{document}